\documentclass[a4paper,11pt]{amsart}

\usepackage{amscd,amsthm,amsfonts,amssymb,amsmath}
\usepackage[colorlinks=true]{hyperref}
\usepackage{fullpage}
\usepackage[all]{xy}

    \newcommand{\BC}{{\mathbb {C}}} 
     \newcommand{\BF}{{\mathbb {F}}}

     \newcommand{\BP}{{\mathbb {P}}}
    \newcommand{\BQ}{{\mathbb {Q}}} \newcommand{\BR}{{\mathbb {R}}}

     \newcommand{\BZ}{{\mathbb {Z}}}

     \newcommand{\CH}{{\mathcal {H}}}

     \newcommand{\fL}{{\mathfrak{L}}}

     \newcommand{\an}{{\mathrm{an}}}

    \newcommand{\Gal}{{\mathrm{Gal}}}

    \newcommand{\ord}{{\mathrm{ord}}}

    \renewcommand{\mod}{\ \mathrm{mod}\ }

    \newcommand{\Sel}{{\mathrm{Sel}}}

      \newcommand{\im}{{\mathrm{im}}}

\DeclareFontFamily{U}{wncy}{}
\DeclareFontShape{U}{wncy}{m}{n}{<->wncyr10}{}
\DeclareSymbolFont{mcy}{U}{wncy}{m}{n}
\DeclareMathSymbol{\Sha}{\mathord}{mcy}{"58}

    \theoremstyle{plain}
    \newtheorem{thm}{Theorem}[section] 
    \newtheorem{lem}[thm]{Lemma}  \newtheorem{prop}[thm]{Proposition}

\theoremstyle{remark} \newtheorem{remark}[thm]{Remark}
\theoremstyle{remark} 
\theoremstyle{remark} 

    \numberwithin{equation}{section}

\begin{document}

\title{On the $2$-part of the Birch and Swinnerton-Dyer conjecture for quadratic twists of elliptic curves}

\author{Li Cai, Chao Li, Shuai Zhai}

%\thanks{XXX was supported by}

\thanks{Li Cai was supported by NSFC (Grants No. 11601255 and 11671380). Chao Li was partially supported by the NSF grant DMS-1802269. Shuai Zhai was supported by NSFC (Grant No. 11601272).}


\begin{abstract}
In the present paper, we prove, for a large class of elliptic curves defined over $\BQ$, the existence of an explicit infinite family of quadratic twists with analytic rank $0$. In addition, we establish the $2$-part of the conjecture of Birch and Swinnerton-Dyer for many of these infinite families of quadratic twists. Recently, Xin Wan has used our results to prove for the first time the full Birch--Swinnerton-Dyer conjecture for some explicit infinite families of elliptic curves defined over $\BQ$ without complex multiplication. 
\end{abstract}

\maketitle

%\tableofcontents

\section{Introduction}

Let $E$ be an elliptic curve defined over $\BQ$ with conductor $C$, and complex $L$-series $L(E,s)$. The Birch and Swinnerton-Dyer conjecture asserts that the rank of $E(\BQ)$ is equal to its analytic rank $r_\an:=\ord_{s=1}L(E,s)$. It furthermore predicts that the Tate--Shafarevich group $\Sha(E)$ is always finite, and that
\begin{equation}\label{fbsd}
\frac{L^{(r_\an)}(E,1)}{{r_\an}!\Omega(E) R(E)}=\frac{\prod_\ell c_\ell(E)\cdot |\Sha(E)|}{|E(\BQ)_\mathrm{tor}|^2},  
\end{equation}
where $\Omega(E)$ is the Tamagawa factor at infinity, $R(E)$ is the regulator formed with the N\'{e}ron--Tate pairing, $E(\BQ)_\mathrm{tor}$ is the torsion subgroup of $E(\BQ)$, and the $c_\ell(E)$ are the Tamagawa factors (see \cite{Tate}, for example). In fact, the finiteness of $\Sha(E)$ is only known at present when $r_\an$ is at most $1$, in which case it is also known that $r_\an$ is equal to the rank of $E(\BQ)$ (see \cite{Gross2}, for example). 

If $p$ is any prime number, the equality of the powers of $p$ occurring on the two sides of \eqref{fbsd} is called the $p$-part of the exact Birch--Swinnerton-Dyer formula (but we should remember that the left hand side of \eqref{fbsd} is only known at present to be a rational number when $r_\an$ is at most $1$). We stress that, up until now, the full Birch--Swinnerton-Dyer conjecture had never been proven for infinitely many elliptic curves without complex multiplication. Roughly speaking, our present knowledge of Iwasawa theory shows that for a given $E$, the $p$-part of the Birch--Swinnerton-Dyer conjecture is valid for all sufficiently large primes $p$ when $r_\an \leq 1$. But there are real technical difficulties at present in using Iwasawa theory to prove, in particular, the $2$-part of the Birch--Swinnerton-Dyer conjecture. However, we can apply rather classical results on modular symbols to derive the precise $2$-adic valuation of the algebraic part of the value of the complex $L$-series at $s=1$ in the family of quadratic twists of certain optimal elliptic curves $E$ over $\BQ$ with $r_\an=0$ and $E(\BQ)[2] \cong \BZ/2\BZ$. In particular, for all of these twists, our results show that $r_\an=0$, whence the Mordell--Weil group and the Tate--Shafarevich group of these twists are both finite by the celebrated theorems of Gross--Zagier \cite{Gross1} and Kolyvagin \cite{Kolyvagin}. Moreover, we can prove the $2$-part of exact Birch--Swinnerton-Dyer formula for some of these twists. Happily, Xin Wan has now used some of our results in this paper, combined with deep arguments from Iwasawa theory to prove for the first time the validity of the full Birch--Swinnerton-Dyer conjecture for infinitely many elliptic curves over $\BQ$ without complex multiplication (see \cite[Appendix]{Wan}). He employs deep and complicated arguments from Iwasawa theory to establish the $p$-part of the Birch--Swinnerton-Dyer conjecture for all odd primes $p$ for the elliptic curves in these families. However, it is still not known how to extend these Iwasawa-theoretic arguments to the prime $p=2$,  whereas our elementary arguments work well for $p=2$. For the current progress on the Birch and Swinnerton-Dyer conjecture, one can see the survey article by Coates \cite{Coates3}.

We now denote the left-hand-side of \eqref{fbsd} by $L^{(alg)}(E,1)$. In particular, when $r_\an=0$, 
$$
L^{(alg)}(E,1):=L(E,1)/\Omega_{E},
$$
where $\Omega_{E}$ is equal to $\Omega_{E}^+$ or $2\Omega_{E}^+$, depending on whether $E(\BR)$ is connected, and here $\Omega_{E}^+$ is the least positive real period of a N\'{e}ron differential on a global minimal Weierstrass equation for $E$. For each discriminant $m$ of a quadratic extension of $\BQ$, we write $E^{(m)}$ for the twist of $E$ by this quadratic extension, and write $L(E^{(m)},s)$ for its complex $L$-series. Let $ord_2$ be the order valuation of $\BQ$ at the prime $2$, normalized by $ord_2(2) = 1$, and with $ord_2(0) = \infty$. If $q$ be any prime of good reduction for $E$, let $a_q$ be the trace of Frobenius at $q$ on $E$, so that $N_q=1+q-a_q$ is the number of $\BF_q$-points on the reduction of $E$ modulo $q$. We shall always assume that $E(\BQ)[2] \cong \BZ/2\BZ$, and we write $E':=E/E(\BQ)[2]$ for the $2$-isogenous curve of $E$. For each integer $n > 1$, write $E[n]$ for the Galois module of $n$-division points on $E$. Let $\mathcal{S}$ be the set of primes 
$$
\mathcal{S}=\{q \equiv 1 \mod 4: q \nmid C, ord_2(N_q)=1\}.
$$ 
\begin{thm}\label{MainThm}
Let $E$ be an optimal elliptic curve over $\BQ$ with conductor $C$. Assume that 
\begin{enumerate}
  \item $E$ has odd Manin constant;
  \item $E(\BQ)[2] \cong \BZ/2\BZ$;
  \item $ord_2 (L^{(alg)}(E,1))=-1$.
\end{enumerate}
Let $M=q_1 q_2 \cdots q_r$ be a product of $r$ distinct primes in $\mathcal{S}$. Then $L(E^{(M)},1) \neq 0$, and we have
$$
 ord_2(L^{(alg)}(E^{(M)},1))=r-1.
$$
In particular, $E^{(M)}(\BQ)$ and $\Sha(E^{(M)})$ are both finite.
\end{thm}

\begin{remark}
This theorem generalizes \cite[Theorem 1.2]{Coates2} (where $E=X_0(49)$) and  \cite[Theorem 1.3]{Cai1} (where $E=X_0(36)$) to a much wider class of elliptic curves $E$, with no hypothesis of complex multiplication. It also generalizes \cite{{Kriz2}, {Zhai}}, where only prime twists are considered. For similar results for $E$ without rational $2$-torsion, see \cite{{Kriz1}, {Zhai}}. In the presence of rational $2$-torsion, the methods of \cite{{Kriz2}, {Zhai}} cannot easily treat twists by non-prime quadratic discriminants, because the obvious induction argument fails. We overcome this difficulty by introducing a new integrality argument to make the induction work.
\end{remark}

\begin{remark}
If $\mathcal{S}$ is non-empty, we must have $E(\BQ)[2] \cong E'(\BQ)[2] \cong \BZ/2\BZ$, which is also equivalent to the assertion that $q$ is inert in  both the $2$-division field $\BQ(E[2])$ and $\BQ(E'[2])$ (see \cite[Lemma 4.1]{Kriz2}), where as before $E':=E/E(\BQ)[2]$. Thus, by Chebotarev's density theorem, the set of primes $\mathcal{S}$ has positive density. 
\end{remark}

\begin{remark}
We suppose that the Manin constant of $E$ has to be odd, which will be fully discussed in Section 2. However, we can remove the Manin constant assumption when $4 \nmid C$ by the recent work of \v{C}esnavi\v{c}ius \cite{Cesnavicius}. Moreover, the conjecture that the Manin constant is always $\pm 1$ has been proved by Cremona for all optimal elliptic curves of conductor less than $390000$ (see \cite{Cremona2}). 
\end{remark}

Our second main result is a proof of the $2$-part of the Birch and Swinnerton-Dyer conjecture for many of the twists in Theorem \ref{MainThm}. As before, let $E':=E/E(\BQ)[2]$ be the $2$-isogenous curve of $E$.

\begin{thm}\label{MainThm-BSD}
Let $E$ and $M$ be as in Theorem \ref{MainThm}. Assume further that
\begin{enumerate}
  \item $\Sha(E')[2]=0$; 
  \item all primes $\ell$ which divide $2C$ split in $\BQ(\sqrt M)$; 
  \item the $2$-part of the Birch and Swinnerton-Dyer conjecture holds for $E$.
\end{enumerate}
Then the $2$-primary component of $\Sha(E^{(M)})$ is zero, and the $2$-part of the Birch and Swinnerton-Dyer conjecture holds for $E^{(M)}$.
\end{thm}

\begin{remark}
In view of our assumption that $\#(E(\BQ)[2])=2$, the $2$-part of the Birch and Swinnerton-Dyer conjecture for $E$ would show that our hypothesis that $ord_2(L^{(alg)}(E,1))=-1$ implies  $\Sha(E)[2]=0$, but it is still not known how to prove this at present. However, if we assume that the $2$-part of the Birch and Swinnerton-Dyer conjecture holds for $E$, as well as the hypotheses of Theorem \ref{MainThm}, we will have $\Sha(E)[2]=0$. Moreover, if we assume two more conditions on $\Sha(E')[2]$ and $\ell$, then we can compare the local conditions of the Selmer groups of $E$ and $E^{(M)}$, and get the triviality of $\Sha(E^{(M)})[2]$.
\end{remark}

\begin{remark}
The $2$-part of Birch--Swinnerton-Dyer conjecture for a single elliptic curve (of small conductor) can be verified by numerical calculation when $r_\an=0$. Theorem \ref{MainThm-BSD} then allows one to deduce the $2$-part of Birch--Swinnerton-Dyer conjecture for many of its quadratic twists (of arbitrarily large conductor).
\end{remark}

\begin{remark}
We emphasize that the theorem applies to elliptic curves with various different reduction types at 2, such as $X_0(14)$ with non-split multiplicative reduction at $2$, `$34A1$' with split multiplicative reduction at $2$, and `$99C1$' with good ordinary reduction at $2$ (we use Cremona's label for each curve). We emphasize that it also applies to elliptic curves with potentially supersingular reduction at $2$, such as $X_0(36)$ and `$56B1$'. We will give a detailed descriptions of quadratic twists of $X_0(14)$ and some numerical examples in Section 6. Of course, the theorem could apply more families of elliptic curves, such as quadratic twists of `$46A1$', $X_0(49)$, `$66A1$',  `$66C1$' and so on.
\end{remark}

Recently, a remarkable preprint of Smith \cite{Smith} uses some arithmetic properties of elliptic curves at the prime $2$ to establish some deep results conjectured by Goldfeld (in particular, that the set of all square free congruent numbers congruent to $1,2,3$ modulo $8$ has natural density zero).  However, Smith's analytic arguments at present seem only valid for elliptic curves with full rational $2$-torsion. We should mention that the non-vanishing result presented in this paper could give a much weaker result in the direction of Goldfeld's conjecture for the family of elliptic curves in Theorem \ref{MainThm}. We also remark that it would be possible to prove analogous results to those established here for rank one quadratic twists of elliptic curves, by combining the Heegner points arguments (see \cite{Coates2}) and the explicit Gross--Zagier formula (see \cite{Cai2}).

\smallskip

%\noindent
\emph{Acknowledgments.} We would like to thank John Coates for very helpful discussions, advices and comments; Jack Thorne for his useful comments and suggestions; and ianya Liu and Ye Tian for their encouragement. We thank the referee for the useful comments. The second-named author (CL) and the third-named author (SZ) would also like to thank X. Wan and the Morningside Center of Mathematics for the hospitality during their visits.

\bigskip

\section{Modular symbols}

Modular symbols were first used by Birch, and a little later by Manin \cite{Manin}. They subsequently became the basic tool in Cremona's construction of his remarkable tables of elliptic curves and their arithmetic invariants \cite{Cremona1}. We shall show in this paper that they are also very useful in studying the 2-part of the conjecture of Birch and Swinnerton-Dyer. We first recall some basic results of modular symbols, for more details, one can see \cite{Zhai}, but we shall give these results as well for reading convenience.

For each integer $C \geq 1$, let $S_2(\Gamma_0(C))$ be the space of cusp forms of weight 2 for $\Gamma_0(C)$. In what follows, $f$ will always denote a normalized primitive eigenform in $S_2(\Gamma_0(C))$, all of whose Fourier coefficients belong to $\BQ$. Thus, $f$ will correspond to an isogeny class of elliptic curves defined over $\BQ$, and we will denote by $E$ the unique {\it optimal} elliptic curve in the $\BQ$-isogeny class of $E$. The complex $L$-series $L(E,s)$ will then coincide with the complex $L$-series attached to the modular form $f$. Moreover, there will be a non-constant rational map defined over $\BQ$
$$
\varphi: X_0(C) \rightarrow E,
$$
which does not factor through any other elliptic curve in the isogeny class of $E$. Let $\omega$ denote a N\'{e}ron differential on a global minimal Weierstrass equation for $E$.
Then, writing $\varphi^*(\omega)$ for the pull back of $\omega$ by $\varphi$, there exists $\nu_E \in \BQ^\times$ such that
\begin{equation}\label{mc}
\nu_Ef(\tau)d\tau = \varphi^*(\omega).
\end{equation}
The rational number $\nu_E$ is called the {\it Manin constant}. It is well known to lie in $\BZ$, and it is conjectured to always be equal to 1. Moreover, it is known to be odd whenever the conductor $C$ of $E$ is odd. Let $\CH$ be the upper half plane, and put $\CH^*=\CH\cup\BP^1(\BQ)$. Let $g$ be any element of  $\Gamma_0(C)$. Let $\alpha, \beta$ be two points in $\CH^*$ such that $\beta=g\alpha$. Then any path from $\alpha$ to $\beta$ on $\CH^*$ is a closed path on $X_0(C)$ whose homology class only depends on $\alpha$ and $\beta$. Hence it determines an integral homology class in $H_1(X_0(C),\BZ)$, and we denote this homology class by the \emph{modular symbol} $\{\alpha,\beta\}$. We can then form the modular symbol
$$
\langle \{\alpha,\beta\}, f \rangle := \int_{\alpha}^{\beta} 2\pi i f(z) dz.
$$
The period lattice $\Lambda_f$ of the modular form $f$ is defined to be the set of these modular symbols for all such pairs  $\{\alpha,\beta\}$. It is a discrete subgroup of $\BC$
of rank 2. If $\fL_E$ denotes the period lattice of a N\'{e}ron differential $\omega$ on $E$, it follows from \eqref{mc}
that
\begin{equation}\label{mc2}
\fL_E = \nu_E\Lambda_f.
\end{equation}
Define $\Omega_E^+$ (respectively, $i \Omega_E^-$) to be the least positive real (respectively, purely imaginary) period of a N\'{e}ron differential of a global minimal equation for $E$, and $\Omega_f^+$ (respectively, $i \Omega_f^-$) to be the least positive real (respectively, purely imaginary) period of $f$. Thus, by \eqref{mc2}, we have
\begin{equation}\label{mc3}
\Omega_E^+= \nu_E\Omega_f^+, \ \ \ \Omega_E^-= \nu_E\Omega_f^-.
\end{equation}
In this section, we will carry out all of our computations with the period lattice $\Lambda_f$, but whenever we subsequently translate them into assertions about the conjecture of Birch and Swinnerton-Dyer for the elliptic curve $E$, we must switch to the period lattice $\fL_E$ by making use of \eqref{mc2}.

\medskip

More generally, if $\alpha, \beta$ are any two elements of $\CH^*$, and $g$ is any element of $S_2(\Gamma_0(C))$, we put $\langle \{\alpha,\beta\}, g \rangle := \int_{\alpha}^{\beta} 2\pi i g(z) dz$.  This linear functional defines an element of $H_1(X_0(C), \BR)$, which we also denote by $\{\alpha, \beta\}.$ 

Let $m$ be a positive integer satisfying $(m,C)=1$. Let $a_m$ be the Fourier coefficient of the modular form $f$ attached to $E$. According to Birch, Manin \cite[Theorem 4.2]{Manin} and Cremona \cite[Chapter 3]{Cremona1}, we have the following formulae:
\begin{equation}\label{ms1}
(\sum_{l \mid m}l-a_m) L(E,1)=-\sum_{\substack{l \mid m \\ k \mod l}} \langle \{0,\frac{k}{l}\}, f \rangle;
\end{equation}
here $l$ runs over all positive divisors of $m$; and
\begin{equation}\label{ms2}
L(E,\chi,1)=\frac{g(\bar{\chi})}{m} \sum_{k \mod m}{\chi}(k) \langle \{0,\frac{k}{m}\}, f \rangle;
\end{equation}
here $\chi$ is any primitive Dirichlet character modulo $m$, and $g(\bar{\chi})= \sum_{k \mod m} \bar{\chi}(k) e^{2\pi i \frac{k}{m}}$.

For each odd square-free positive integer $m$, we define $r(m)$ to be the number of prime factors of $m$. Also, in what follows, we shall always only consider the positive divisors of $m$, and define $\chi_m$ to be the primitive quadratic character modulo $m$. Define
$$
S_m:=\sum_{k=1}^{m} \langle \{0,\frac{k}{m}\}, f \rangle, \, \, \, S'_m:=\sum_{\substack{k=1 \\ (k,m)=1}}^{m} \langle \{0,\frac{k}{m}\}, f \rangle, \, \, \, T_m:=\sum_{k=1}^{m} \chi_m(k) \langle \{0,\frac{k}{m}\}, f \rangle.
$$
Recall that (see \cite[Lemma 2.2]{Zhai}), for each odd square-free positive integer $m>1$, we have
\begin{equation}\label{ms_sum_Sl}
\sum_{l \mid m} S_l = \sum_{d=1}^{r(m)} 2^{r(m)-d} \sum_{\substack{n \mid m \\ r(n)=d}} S'_n.
\end{equation}
We repeatedly use the above identity to prove the following lemma.

\begin{lem}\label{Nq_L(E,1)}
Let $E$ be the optimal elliptic curve over $\BQ$ attached to $f$. Let $m$ be any integer of the form $m=q_1 q_2 \cdots q_{r(m)}$, with $(m,C)=1$, $r(m) \geq 1$, and $q_1, \ldots, q_{r(m)}$ arbitrary distinct odd primes. Then we have
$$
N_{q_1} N_{q_2} \cdots N_{q_{r(m)}} L(E,1) = \sum_{d=1}^{r(m)} \sum_{\substack{n \mid m \\ r(n)=d}} b_n S'_n,
$$
where $b_n = (-1)^{r(m)} \prod_{q \mid \frac{m}{n}} (1 - q)$, here $q$ runs over the prime factors of $\frac{m}{n}$.
\end{lem}

\begin{proof}
We give the proof of the lemma by induction on $r(m)$, the number of prime factors of $m$. The assertion is true for $r(m)=1$ by \eqref{ms1}. Assume next that $r(m)=2$. Note that
\begin{align*}
N_{q_1} N_{q_2} &= - ((1+q_1)(1+q_2) - (1+q_1-N_{q_1})(1+q_2-N_{q_2})) + (1+q_2)N_{q_1}+(1+q_1)N_{q_2} \\
                &= - ((1+q_1)(1+q_2) - a_{q_1}a_{q_2}) + (1+q_2)N_{q_1}+(1+q_1)N_{q_2},
\end{align*}
and in view of \eqref{ms1} and \eqref{ms_sum_Sl}, we then have that
\begin{align*}
N_{q_1} N_{q_2} L(E,1) &= \sum_{l \mid q_1 q_2} S_l - ((1+q_2)S_{q_1}+(1+q_1)S_{q_2}) \\
                       &= (1-q_2)S_{q_1}+(1-q_1)S_{q_2} + S'_{q_1 q_2},
\end{align*}
as required. Now assume  $r(m) > 2$, and that the lemma is true for all divisors $n>1$ of $m$ with $n \neq m$. We then consider the case $m=q_1 q_2 \cdots q_{r(m)}$. First note that
\begin{align*}
N_{q_1} N_{q_2} \cdots N_{q_{r(m)}} =& (-1)^{r(m)-1}((1+q_1)(1+q_2) \cdots (1+q_{r(m)}) - a_{q_1}a_{q_2} \cdots a_{q_{r(m)}}) \\
                                     & + (-1)^{r(m)-2}\sum_{i=1}^{r(m)}N_{q_i} \prod_{\substack{k=1 \\ k \neq i}}^{r(m)}(1+q_k) + (-1)^{r(m)-3}\sum_{i,j=1}^{r(m)}N_{q_i}N_{q_j} \prod_{\substack{k=1 \\ k \neq i,j}}^{r(m)}(1+q_k) \\
                                     & + \cdots + (-1)\sum_{i,j=1}^{r(m)}(1+q_i)(1+q_j) \prod_{\substack{k=1 \\ k \neq i,j}}^{r(m)}N_{q_k}
                                     + \sum_{i=1}^{r(m)}(1+q_i) \prod_{\substack{k=1 \\ k \neq i}}^{r(m)}N_{q_k}.
\end{align*}
Without loss of generality, here we can just consider the coefficients of $S'_{q_1}$, $S'_{q_1 q_2}$, $\ldots$, $S'_{q_1 q_2 \cdots q_{r(m)}}$ in the identity of the lemma, i.e. $b_{q_1}$, $b_{q_1 q_2}$, $\ldots$, $b_{q_1 q_2 \cdots q_{r(m)}}$. By our assumption, and again in view of \eqref{ms1} and \eqref{ms_sum_Sl}, we conclude that
\begin{align*}
b_{q_1} =& -(-1)^{r(m)-1}2^{r(m)-1} + (-1)^{r(m)-1}\prod_{i=2}^{r(m)}(1+q_i)
          + (-1)^{r(m)-1}\sum_{i=2}^{r(m)}(1-q_i) \prod_{\substack{k=2 \\ k \neq i}}^{r(m)}(1+q_k) \\
          &+ (-1)^{r(m)-1}\sum_{\substack{i,j=2 \\ i \neq j}}^{r(m)}(1-q_i)(1-q_j)\prod_{\substack{k=2 \\ k \neq i,j}}^{r(m)}(1+q_k) + \cdots + (-1)^{r(m)-1}\sum_{i=2}^{r(m)}(1+q_i) \prod_{\substack{k=2 \\ k \neq i}}^{r(m)}(1-q_k).
\end{align*}
Note that
$$
-2^{r(m)-1} = - \prod_{i=2}^{r(m)}((1-q_i) + (1+q_i)),
$$
hence we have
$$
b_{q_1} = (-1)^{r(m)} \prod_{i=2}^{r(m)}(1-q_i).
$$
Similar arguments hold for $b_{q_1 q_2}, \ldots, b_{q_1 q_2 \cdots q_{r(m)-1}}$, and it is easy to see that
$$
b_{q_1 q_2 \cdots q_{r(m)}} = (-1)^{r(m)}.
$$
The proof of the lemma is complete.
\end{proof}

\begin{lem}\label{ord2_S'm}
Let $E$ be the optimal elliptic curve over $\BQ$ with analytic rank zero attached to $f$. Let $m$ be any integer of the form $m=q_1 q_2 \cdots q_{r(m)}$, with $(m,C)=1$, $r(m) \geq 1$, and $q_1, \ldots, q_{r(m)}$ arbitrary distinct odd primes congruent to $1$ modulo $4$. If $ord_2 (N_{q_i}) = 1$ holds for any $1 \leq i \leq r(m)$, then we have
$$
ord_2 (S'_m / \Omega_f^+) = ord_2 (N_{q_1} N_{q_2} \cdots N_{q_{r(m)}} L(E,1) / \Omega_f^+).
$$
\end{lem}
\begin{proof}
We give the proof of the lemma by induction on $r(m)$. The assertion is obviously true for $r(m)=1$ according to \eqref{ms1}. When $r(m)=2$, say $m=q_1 q_2$, by Lemma \ref{Nq_L(E,1)}, we have that
$$
N_{q_1} N_{q_2} L(E,1) =  (1 - q_2) S'_{q_1} + (1 - q_1) S'_{q_2} + S'_{q_1 q_2}.
$$
The assertion for $r(m)=2$ then follows by noting that $q_i \equiv 1 \mod 4$ and the induction assumption. Now assume  $r(m) > 2$, and that the lemma is true for all divisors $n>1$ of $m$ with $n \neq m$. We then consider the case $m=q_1 q_2 \cdots q_{r(m)}$. According to Lemma \ref{Nq_L(E,1)}, we have that
$$
N_{q_1} N_{q_2} \cdots N_{q_{r(m)}} L(E,1) = \sum_{d=1}^{r(m)-1} \sum_{\substack{n \mid m \\ r(n)=d}} (-1)^{r(m)} \prod_{q \mid \frac{m}{n}} (1 - q) S'_n + (-1)^{r(m)} S'_m.
$$
By our assumption, it is not difficult to see that
$$
ord_2(\prod_{q \mid \frac{m}{n}} (1 - q) S'_n / \Omega_f^+) > ord_2(N_{q_1} N_{q_2} \cdots N_{q_{r(m)}} L(E,1)/ \Omega_f^+)
$$
holds for all divisors $n>1$ of $m$ with $n \neq m$. Then the assertion for $m=q_1 q_2 \cdots q_{r(m)}$ follows immediately. This completes the proof of the lemma.
\end{proof}

\bigskip

\section{Integrality at 2}

Let $E$ be the optimal elliptic curve defined over $\BQ$ with discriminant $\Delta_E$ and conductor $C$, which is attached to our modular form $f$. In this section, we will prove some results of integrality at $2$, and apply them to get the non-vanishing results for some certain quadratic twists of elliptic curves, provided $L(E,1) \neq 0$. Recall that 
$$
\Omega_E^+ = \nu_E\Omega_f^+,
$$
we then have
$$
ord_2(L(E,1)/\Omega_E^+) = ord_2(L(E,1)/\Omega_f^+) - ord_2(\nu_E) = ord_2(L(E,1)/\Omega_f^+),
$$
under our assumption on the Manin constant.

When the complex $L$-series of $E$ does not vanish at $s=1$, for every prime number $p$, the strong Birch--Swinnerton-Dyer conjecture predicts the following exact formula
$$
ord_p (L^{(alg)}(E,1)) = ord_p (\#(\Sha(E))) + ord_p (\prod_{\ell |C} c_{\ell}(E)) - 2ord_p(\#(E(\BQ))),
$$
We begin by establishing some preliminary results, which will be needed for the proof of the desired results. Throughout this section, we will always assume $m \equiv 1 \mod 4$. Since the form of the period lattice of a N\'{e}ron differential on $E$ is different, according to the sign of the discriminant of $E$, we first consider the case when the discriminant of $E$ is negative.

Recall that when the discriminant of $E$ is negative, then $E(\BR)$ has only one real component, and so the period lattice $\fL$ of a N\'{e}ron differential on $E$ has a $\BZ$-basis of the form
$$
\left[\Omega_E^+, \frac{\Omega_E^+ + i \Omega_E^-}{2}\right],
$$
where $\Omega_E^+$ and $\Omega_E^-$ are both real, and the period lattice $\Lambda_f$ of $f$ has a $\BZ$-basis of the form
$$
\left[\Omega_f^+, \frac{\Omega_f^+ + i \Omega_f^-}{2}\right],
$$
where $\Omega_f^+$ and $\Omega_f^-$ are also both real. We can then write
\begin{equation}\label{k/mf-1}
\langle \{0,\frac{k}{m}\}, f \rangle = (s_{k,m} \Omega_f^+ + i t_{k,m} \Omega_f^-)/2
\end{equation}
for any integer $m$ coprime to $C$, where $s_{k,m}, t_{k,m}$ are integers of the same parity. Moreover, by the basic property of modular symbols, $\langle \{0,\frac{k}{m}\}, f \rangle$ and $\langle \{0,\frac{m-k}{m}\}, f \rangle$ are complex conjugate periods of $f$. Thus we obtain
\begin{equation}\label{S'm}
S'_m/\Omega_f^+=\sum_{\substack{k=1 \\ (k,m)=1}}^{(m-1)/2} s_{k,m}.
\end{equation}
Similarly, when $m \equiv 1 \mod 4$, we have
\begin{equation}\label{Tm}
T_m/\Omega_f^+=\sum_{\substack{k=1 \\ (k,m)=1}}^{(m-1)/2} \chi_{m}(k) s_{k,m}.
\end{equation}
Moreover, in this case, by \eqref{S'm}, we always have that
$$
ord_2(N_q L(E,1) / {\Omega_f^+}) \geq 0,
$$
for any prime $q$ with $(q,C)=1$. We define
$$
T'_{d,m} = \sum_{k \in (\BZ/m\BZ)^{\times}} \chi_d(k) \langle \{0,\frac{k}{m}\}, f \rangle,
$$
then we have the following theorem of integrality at $2$.

\begin{thm}\label{Integrality-1}
Let $E$ be an optimal elliptic curve over $\BQ$ with $\Delta_E<0$. Let $m$ be any integer of the form $m=q_1 q_2 \cdots q_{r(m)}$, with $(m,C)=1$, $r(m) \geq 1$, and $q_1, \ldots, q_{r(m)}$ arbitrary distinct odd primes in $\mathcal{S}$. Then 
$$
\sum_{d \mid m} T'_{d,m}/{\Omega_f^+} = 2^{r(m)} \Psi_m,
$$
where $\Psi_m$ is an integer.
\end{thm}
\begin{proof}
It is easy to see that
\begin{align*}
\sum_{d \mid m} T'_{d,m} &= \sum_{k \in (\BZ/m\BZ)^{\times}} \sum_{d \mid m} \chi_d(k) \langle \{0,\frac{k}{m}\}, f \rangle  \\
&= 2^{r(m)} \mathop{{\sum}^*}_{k \in (\BZ/m\BZ)^{\times}} \langle \{0,\frac{k}{m}\}, f \rangle,
\end{align*}
where $\mathop{{\sum}^*}$ means that $k$ runs over all the elements in $(\BZ/m\BZ)^{\times}$ such that $\chi_{q_i}(k)=1$ for all $1 \leq i \leq r(m)$. Since $q_i \equiv 1 \mod 4$, if $k$ is of an element in the above summation, so is $m-k$. Then by \eqref{k/mf-1}, we have that 
$$
\mathop{{\sum}^*}_{k \in (\BZ/m\BZ)^{\times}} \langle \{0,\frac{k}{m}\}, f \rangle = \mathop{{\sum}^*}_{\substack{k=1 \\ (k,m)=1}}^{(m-1)/2} s_{k,m} \Omega_f^+.
$$
Then the argument follows immediately if we define 
$$
\Psi_m= \mathop{{\sum}^*}_{\substack{k=1 \\ (k,m)=1}}^{(m-1)/2} s_{k,m},
$$
which is an integer.
\end{proof}

When the discriminant of $E$ is positive, then $E(\BR)$ has two real components, and so the period lattice $\fL$ of a N\'{e}ron differential on $E$ has a $\BZ$-basis of the form
$$
[\Omega_E^+, i \Omega_E^-],
$$
with $\Omega_E^+$ and $\Omega_E^-$ real numbers, and the period lattice $\Lambda_f$ of $f$ has a $\BZ$-basis of the form
$$
[\Omega_f^+, i \Omega_f^-],
$$
with $\Omega_f^+$ and $\Omega_f^-$ real numbers too. We can then write
\begin{equation}\label{k/mf-2}
\langle \{0,\frac{k}{m}\}, f \rangle = s_{k,m} \Omega_f^+ + i t_{k,m} \Omega_f^-
\end{equation}
for any integer $m$ coprime to $C$, where $s_{k,m}, t_{k,m}$ are integers. Similarly, we can obtain
\begin{equation}\label{S'm'}
S'_m/\Omega_f^+=2\sum_{\substack{k=1 \\ (k,m)=1}}^{(m-1)/2} s_{k,m},
\end{equation}
and when $m \equiv 1 \mod 4$, we have
\begin{equation}\label{Tm'}
T_m/\Omega_f^+=2\sum_{\substack{k=1 \\ (k,m)=1}}^{(m-1)/2} \chi_{m}(k) s_{k,m}.
\end{equation}
Moreover, in this case, by \eqref{S'm'}, we always have that
$$
ord_2(N_q L(E,1) / {\Omega_f^+}) \geq 1,
$$
for any prime $q$ with $(q,C)=1$. We then have the following parallel theorem of integrality at $2$.

\begin{thm}\label{Integrality-2}
Let $E$ be an optimal elliptic curve over $\BQ$ with $\Delta_E>0$. Let $m$ be any integer of the form $m=q_1 q_2 \cdots q_{r(m)}$, with $(m,C)=1$, $r(m) \geq 1$, and $q_1, \ldots, q_{r(m)}$ arbitrary distinct odd primes in $\mathcal{S}$.  Then 
$$
\sum_{d \mid m} T'_{d,m}/{\Omega_f^+} = 2^{r(m)+1} \Psi_m,
$$
where $\Psi_m$ is an integer.
\end{thm}
\begin{proof}
The proof of the above theorem is similar to Theorem \ref{Integrality-1}. As usual, we have
\begin{align*}
\sum_{d \mid m} T'_{d,m} &= \sum_{k \in (\BZ/m\BZ)^{\times}} \sum_{d \mid m} \chi_d(k) \langle \{0,\frac{k}{m}\}, f \rangle  \\
&= 2^{r(m)} \mathop{{\sum}^*}_{k \in (\BZ/m\BZ)^{\times}} \langle \{0,\frac{k}{m}\}, f \rangle,
\end{align*}
where $\mathop{{\sum}^*}$ means that $k$ runs over all the elements in $(\BZ/m\BZ)^{\times}$ such that $\chi_{q_i}(k)=1$ for all $1 \leq i \leq r(m)$. Since $q_i \equiv 1 \mod 4$, if $k$ is of an element in the above summation, so is $m-k$. But when the discriminant is positive, by \eqref{k/mf-2}, we have 
$$\mathop{{\sum}^*}_{k \in (\BZ/m\BZ)^{\times}} \langle \{0,\frac{k}{m}\}, f \rangle = 2 \mathop{{\sum}^*}_{\substack{k=1 \\ (k,m)=1}}^{(m-1)/2} s_{k,m} \Omega_f^+.
$$
Then the argument follows immediately if we define 
$$
\Psi_m= \mathop{{\sum}^*}_{\substack{k=1 \\ (k,m)=1}}^{(m-1)/2} s_{k,m},
$$
which is an integer.
\end{proof}

\bigskip

\section{Non-vanishing results}

The aim of this section is to apply the results of integrality at $2$ in the previous section to obtain the corresponding non-vanishing results of quadratic twists of elliptic curves. Specifically, we prove the precise $2$-adic valuation of the algebraic central value of these $L$-functions attached to some certain families of quadratic twists of elliptic curves. Moreover, one can use these non-vanishing theorems to verify the $2$-part of the Birch and Swinnerton-Dyer conjecture. Throughout this section, we will always assume $m>0$ and $m \equiv 1 \mod 4$.

Before proving our non-vanishing results, we will first prove the following lemma, in which the action of Hecke operator on modular symbols is involved. For each prime $p$ not dividing the conductor $C$, the Hecke operator $\mathbb{T}_p$ acts on modular symbols $\{\alpha, \beta\}$ via 
$$
\mathbb{T}_p (\{\alpha, \beta\}) = \{p\alpha, p\beta\} + \sum_{k \mod p} \{\frac{\alpha+k}{p}, \frac{\beta+k}{p}\}.
$$
In particular, we have 
$$
 \langle \mathbb{T}_p (\{\alpha, \beta\}), f \rangle = \langle \{\alpha, \beta\}, \mathbb{T}_p f \rangle = a_p \langle \{\alpha, \beta\}, f \rangle,
$$
since $\mathbb{T}_p f= a_p f$.

\begin{lem}\label{T'_{d,m}}
Let $m$ be any integer of the form $m=q_1 q_2 \cdots q_{r(m)}$, with $(m,C)=1$, $r(m) \geq 2$, and $q_1, \ldots, q_{r(m)}$ arbitrary distinct odd primes. Let $d>1$ be a positive integer dividing $m$ and $q$ be a prime dividing $\frac{m}{d}$, then we have
$$
T'_{d,m}  = (a_q - 2 \chi_d(q))T'_{d,\frac{m}{q}} .
$$
\end{lem}
\begin{proof}
Recall that 
$$
T'_{d,m} = \sum_{k \in (\BZ/m\BZ)^{\times}} \chi_d(k) \langle \{0,\frac{k}{m}\}, f \rangle.
$$
By the Chinese remainder theorem, we have 
$$
(\BZ/m\BZ)^{\times} \cong (\BZ/q\BZ)^{\times} \times (\BZ/\frac{m}{q}\BZ)^{\times}.
$$
So we can write
\begin{equation}\label{T'_{d,m}-2}
T'_{d,m} = \sum_{k' \in (\BZ/\frac{m}{q}\BZ)^{\times}} \chi_d(k') \sum_{k \in \BZ/q\BZ} \langle \{0,\frac{\frac{m}{q}k+k'}{m}\}, f \rangle - \sum_{k' \in (\BZ/\frac{m}{q}\BZ)^{\times}} \chi_d(k'q) \langle \{0,\frac{k'q}{m}\}, f \rangle.
\end{equation}
Let the Hecke operator $\mathbb{T}_q$ act on the modular symbol $\{0,\frac{k'}{m/q}\}$, we get that
$$
\mathbb{T}_q (\{0,\frac{k'}{m/q}\}) =  \{0, \frac{k'}{m}\} + \sum_{k \mod q} \{0,\frac{\frac{k'}{m/q}+k}{q}\} - \sum_{k \mod q} \{0, \frac{k}{q}\}.
$$
Hence,
$$
\sum_{k \in \BZ/q\BZ} \langle \{0,\frac{\frac{m}{q}k+k'}{m}\}, f \rangle = a_q \langle \{0,\frac{k'}{m/q}\}, f \rangle + \sum_{k \in \BZ/q\BZ} \langle \{0,\frac{k}{q}\}, f \rangle - \langle \{0,\frac{k'}{m}\}, f \rangle.
$$
Then the first term of the right-hand side of \eqref{T'_{d,m}-2} becomes
$$
\sum_{k' \in (\BZ/\frac{m}{q}\BZ)^{\times}} \chi_d(k')  (a_q \langle \{0,\frac{k'}{m/q}\}, f \rangle + \sum_{k \in \BZ/q\BZ} \langle \{0,\frac{k}{q}\}, f \rangle - \langle \{0,\frac{k'}{m}\}, f \rangle),
$$
which is equal to 
\begin{equation}\label{eq4.1-1}
\sum_{k \in (\BZ/\frac{m}{q}\BZ)^{\times}} \chi_d(k)  (a_q \langle \{0,\frac{k}{m/q}\}, f \rangle - \langle \{0,\frac{k}{m}\}, f \rangle),
\end{equation}
since 
$$
\sum_{k' \in (\BZ/\frac{m}{q}\BZ)^{\times}} \chi_d(k') = 0.
$$
Then \eqref{eq4.1-1} becomes 
$$
a_q \sum_{k \in (\BZ/\frac{m}{q}\BZ)^{\times}} \chi_d(k) \langle \{0,\frac{k}{m/q}\}, f \rangle - \sum_{k' \in (\BZ/\frac{m}{q}\BZ)^{\times}} \chi_d(k'q) \langle \{0,\frac{k'q}{m}\}, f \rangle
$$
if we substitute $k=k'q$ in the second term. We then have 
$$
T'_{d,m} = a_q \sum_{k \in (\BZ/\frac{m}{q}\BZ)^{\times}} \chi_d(k) \langle \{0,\frac{k}{m/q}\}, f \rangle - 2 \chi_d(q) \sum_{k' \in (\BZ/\frac{m}{q}\BZ)^{\times}} \chi_d(k') \langle \{0,\frac{k'}{m/q}\}, f \rangle.
$$
This completes the proof of the lemma by noting that
$$
T'_{d,\frac{m}{q}} = \sum_{k \in (\BZ/\frac{m}{q}\BZ)^{\times}} \chi_d(k) \langle \{0,\frac{k}{m/q}\}, f \rangle.
$$
\end{proof}

Now we are ready to prove Theorem \ref{MainThm}. When the discriminant of $E$ is negative, we have the following result.
\begin{thm}\label{MainThm-1}
Let $E$ be an optimal elliptic curve over $\BQ$ with conductor $C$, and with odd Manin constant. Assume that $E$ has negative discriminant, and satisfies $E(\BQ)[2] \neq 0$ and $ord_2(L(E,1)/{\Omega_f^+}) = -1$. Let $m$ be any integer of the form $m=q_1 q_2 \cdots q_{r(m)}$, with $r(m) \geq 1$ and $q_1, \ldots, q_{r(m)}$ arbitrary distinct odd primes congruent to $1$ modulo $4$, and with $(m,C)=1$. If $ord_2(N_{q_i})=1$ for $1 \leq i \leq r(m)$, then $L(E^{(m)},1) \neq 0$, and we have
$$
ord_2(L(E^{(m)},1)/{\Omega_{E^{(m)}}^+})=r(m)-1.   
$$
\end{thm}
\begin{proof}
We will prove the theorem by induction on $r(m)$, of course we have got the argument when $r(m)=1$ in \cite[Theorem 1.5]{Zhai}. We first note that
\begin{align*}
\sum_{d \mid m} T'_{d,m} &= \sum_{d \mid m} \sum_{k \in (\BZ/m\BZ)^{\times}} \chi_d(k) \langle \{0,\frac{k}{m}\}, f \rangle \\
&= S'_m + \sum_{\substack{d \mid m \\ 1 < d< m}} \sum_{k \in (\BZ/m\BZ)^{\times}} \chi_d(k) \langle \{0,\frac{k}{m}\}, f \rangle + T_m.
\end{align*}
By Lemma \ref{T'_{d,m}}, it is easy to see that
$$
T'_{d,m}  = \prod_{q \mid \frac{m}{d}}(a_q - 2 \chi_d(q)) \cdot T'_{d,d} = \prod_{q \mid \frac{m}{d}}(a_q - 2 \chi_d(q)) \cdot T_d.
$$
Hence,
$$
\sum_{\substack{d \mid m \\ 1 < d< m}} \sum_{k \in (\BZ/m\BZ)^{\times}} \chi_d(k) \langle \{0,\frac{k}{m}\}, f \rangle = \sum_{\substack{d \mid m \\ 1 < d< m}} \prod_{q \mid \frac{m}{d}}(a_q - 2 \chi_d(q)) \cdot T_d.
$$
We then apply Theorem \ref{Integrality-1} and get the following equation
$$
S'_m/{\Omega_f^+} + \sum_{\substack{d \mid m \\ 1 < d< m}} \prod_{q \mid \frac{m}{d}}(a_q - 2 \chi_d(q)) \cdot T_d/{\Omega_f^+} + T_m/{\Omega_f^+} = 2^{r(m)} \Psi_m,
$$
where $\Psi_m$ is an integer with $ord_2 (\Psi_m) \geq 0$.
Note that $ord_2(L(E,1)/{\Omega_f^+}) = -1$ and $ord_2(N_{q_i})=1$, we then have
$$
ord_2(S'_m/{\Omega_f^+} ) = r(m)-1
$$
by Lemma \ref{ord2_S'm}. 
Now assume  $r(m) \geq 2$, and that this theorem has been proved for all products of less than $r(m)$ such primes $q_i$, and note that we have assumed the Manin constant is odd, so we have that 
$$
ord_2(T_d/{\Omega_f^+})  = r(d) -1,
$$
with $1<d<m$ and $d|m$. Moreover, we also have $ord_2(a_q - 2 \chi_d(q))=1$.  Consequently, we have that
$$
ord_2(\prod_{q \mid \frac{m}{d}}(a_q - 2 \chi_d(q)) \cdot T_d/{\Omega_f^+}) = r(m)-1.
$$
Hence we have that
$$
ord_2(\sum_{\substack{d \mid m \\ 1 < d< m}} \prod_{q \mid \frac{m}{d}}(a_q - 2 \chi_d(q)) \cdot T_d/{\Omega_f^+}) = r(m), 
$$
by noting that the number of the terms in this summation is even. So we must have 
$$
ord_2(T_m/{\Omega_f^+}) = r(m) -1,
$$
that is
$$
ord_2(L(E^{(m)},1)/{\Omega_{E^{(m)}}^+}) = r(m)-1.
$$
This completes the proof of this theorem.
\end{proof}

When the discriminant of $E$ is positive, we have the following parallel result.
\begin{thm}\label{MainThm-2}
Let $E$ be an optimal elliptic curve over $\BQ$ with conductor $C$, and with odd Manin constant. Assume that $E$ has positive discriminant and satisfies $E(\BQ)[2] \neq 0$ and $ord_2(L(E,1)/{\Omega_f^+}) = 0$. Let $m$ be any integer of the form $m=q_1 q_2 \cdots q_{r(m)}$, with $r(m) \geq 1$ and $q_1, \ldots, q_{r(m)}$ arbitrary distinct odd primes congruent to $1$ modulo $4$, and with $(m,C)=1$. If $ord_2(N_{q_i})=1$ for $1 \leq i \leq r(m)$, then $L(E^{(m)},1) \neq 0$, and we have 
$$
ord_2(L(E^{(m)},1)/{\Omega_{E^{(m)}}^+})=r(m).
$$
\end{thm}
\begin{proof}
We will also prove the theorem by induction on $r(m)$, of course we have got the argument when $r(m)=1$ in \cite[Theorem 1.7]{Zhai}. Note that $ord_2(L(E,1)/{\Omega_f^+}) = 0$ and $ord_2(N_{q_i})=1$, we then have
$$
ord_2(S'_m/{\Omega_f^+} ) = r(m)
$$
by Lemma \ref{ord2_S'm}. 
Now assume  $r(m) \geq 2$, and that this theorem has been proved for all products of less than $r(m)$ such primes $q_i$, and note that we have assumed the Manin constant is odd, so we have that 
$$
ord_2(T_d/{\Omega_f^+})  = r(d),
$$
with $1<d<m$ and $d|m$. Moreover, we also have $ord_2(a_q - 2 \chi_d(q))=1$. Consequently, we have that
$$
ord_2(\prod_{q \mid \frac{m}{d}}(a_q - 2 \chi_d(q)) \cdot T_d/{\Omega_f^+}) = r(m).
$$
Hence we have that
$$
ord_2(\sum_{\substack{d \mid m \\ 1 < d< m}} \prod_{q \mid \frac{m}{d}}(a_q - 2 \chi_d(q)) \cdot T_d/{\Omega_f^+}) = r(m)+1, 
$$
by noting that the number of the terms in this summation is even. So we must have 
$$
ord_2(T_m/{\Omega_f^+}) = r(m),
$$
by the following equation 
$$
S'_m/{\Omega_f^+} + \sum_{\substack{d \mid m \\ 1 < d< m}} \prod_{q \mid \frac{m}{d}}(a_q - 2 \chi_d(q)) \cdot T_d/{\Omega_f^+} + T_m/{\Omega_f^+} = 2^{r(m)+1} \Psi_m,
$$
which is deduced from Theorem \ref{Integrality-2}.
Hence we have 
$$
ord_2(L(E^{(m)},1)/{\Omega_{E^{(m)}}^+}) = r(m).
$$
This completes the proof of this theorem.
\end{proof}

This completes the proof of Theorem \ref{MainThm} by combining the above two theorems and the celebrated theorems of Gross--Zagier and Kolyvagin.

\bigskip

\section{2-part of the Birch--Swinnerton-Dyer conjecture}

In this section, we will prove that the $2$-part of the Birch--Swinnerton-Dyer conjecture holds for some certain families of the quadratic twists of elliptic curves in the previous section. In particular, we will prove the following result, combining with the non-vanishing result in Theorem \ref{MainThm}, to give a proof of Theorem \ref{MainThm-BSD}.

\begin{prop} \label{prop:bsd2}  
Let $E$ be an elliptic curve over $\BQ$ with $E(\BQ)[2] \cong \BZ/2\BZ$. Let $M=q_1\cdots q_r$ be a square free product of $r$ primes in $\mathcal{S}$.
  \begin{enumerate}
  \item Then $E^{(M)}(\BQ)[2] \cong \BZ/2\BZ$.
  \item Let $\ell_0|C$ be the prime such that $ord_2(c_{\ell_0}(E))=1$. Assume that $\ell_0$ splits in $\BQ(\sqrt{M})$ and $ord_2(\prod_\ell c_\ell(E))=1$. Then 
$$
ord_2(c_\ell(E^{(M)}))=
    \begin{cases}
      0, & \text{if } \ell\ne\ell_0, \text{ and }\ell\nmid M,\\
      1, & \text{if } \ell=\ell_0, \text{ or } \ell|M.
    \end{cases}
$$
In particular, $ord_2(\prod_\ell c_\ell(E^{(M)}))=r+1$.
  \item Assume further that $\Sel_2(E)[2]=\BZ/2\BZ$ and $\Sha(E')[2]=0$. If all primes $\ell|2C$ split in $\BQ(\sqrt{M})$, then $1\le\dim\Sel_2(E^{(M)})\le2$. In particular, if $\Sha(E^{(M)})$ is finite, then $\Sha(E^{(M)})[2]=0$ and $\Sel_2(E^{(M)})=\BZ/2\BZ$.
  \end{enumerate}
\end{prop}

\begin{proof}
(1) It follows from the facts that $E[2]\cong E^{(M)}[2]$ as $G_\BQ$-modules and $E(\BQ)[2]\cong \BZ/2\BZ$.

(2) First consider $\ell\ne \ell_0$ and $\ell\nmid M$. Let $\mathcal{E}$ and $\mathcal{E}^{(M)}$ be the N\'{e}ron model over $\BZ_\ell$ of $E$ and $E^{(M)}$ respectively. Note that $E^{(M)}/\BQ_\ell$ is the unramified quadratic twist of $E^{(M)}$. Since N\'{e}ron models commute with unramified base change, we know that the component groups $\Phi_\mathcal{E}$ and $\Phi_{\mathcal{E}^{(M)}}$ are quadratic twists of each other as $\Gal(\overline{\mathbb{F}}_\ell/\mathbb{F}_\ell)$-modules. In particular, $\Phi_\mathcal{E}[2]\cong\Phi_{\mathcal{E}^{(M)}}[2]$ as $\Gal(\overline{\mathbb{F}}_\ell/\mathbb{F}_\ell)$-modules and thus $$\Phi_\mathcal{E}(\mathbb{F}_\ell)[2]\cong\Phi_{\mathcal{E}^{(M)}}(\mathbb{F}_\ell)[2].$$ It follows that $c_\ell(E)$ and $c_\ell(E^{(M)})$ have the same parity, and hence $c_\ell(E^{(M)})$ is odd.
    
Next consider $\ell | M$. Since $E^{(M)}$ has additive reduction at $\ell$ and $\ell$ is odd, we know that $$\Phi_{\mathcal E^{(M)}}(\mathbb{F}_\ell)[2]\cong E^{(M)}(\BQ_\ell)[2].$$ On the other hand, $E^{(M)}(\BQ_\ell)[2]\cong E(\BQ_\ell)[2]\cong E(\mathbb{F}_\ell)[2]$, which is $\BZ/ 2 \BZ$ since $\ell\in \mathcal{S}$. Thus $ord_2(c_\ell(E^{(M)}))=1$ for any $\ell|M$.

Finally consider $\ell=\ell_0$. By our extra assumption that $\ell_0$ is split in $\BQ(\sqrt{M})$, we know that $E^{(M)}/\BQ_\ell$ and $E/\BQ_\ell$ are isomorphic, hence $c_\ell(E^{(M)})=c_\ell(E)$, which has 2-adic valuation 1.

(3) Let $\phi: E \to E'$ be the isogeny of degree $2$, and $\hat\phi: E' \to E$ be the dual isogeny. We use the following well-known exact sequence relating the 2-Selmer group and $\phi$, $\hat \phi$-Selmer groups (see \cite[Lemma 6.1]{Schaefer}):
$$
0 \rightarrow E'(\BQ)[\hat\phi]/\phi(E(\BQ)[2]) \rightarrow \Sel_{\phi}(E) \rightarrow \Sel_{2}(E) \rightarrow \Sel_{\hat{\phi}}(E') \rightarrow\Sha(E')[\hat\phi]/\phi(\Sha(E)[2]) \rightarrow 0.
$$
By our assumption $\Sel_2(E)=\BZ/2\BZ$ and $\Sha(E')[2]=0$, it follows from the above exact sequence that 
$$
\Sel_\phi(E) \cong E'(\BQ)[\hat\phi]/\phi(E(\BQ)[2])\cong \BZ/2\BZ,\quad \Sel_{\hat \phi}(E')\cong \Sel_2(E)\cong\BZ/2\BZ.
$$ 
By abuse of notation we denote the 2-isogeny $E^{(M)} \rightarrow E'^{(M)}$ again by $\phi$ (note that ${E^{(M)}}'=E'^{(M)}$).

We first claim that the isomorphism of $G_\BQ$-representations $E^{(M)}[\phi]\cong E[\phi]$ induces an isomorphism of $\phi$-Selmer groups 
$$
\Sel_\phi(E^{(M)})\cong \Sel_{\phi}(E).
$$
For $v$ a place of $\BQ$, we denote  the local condition defining the $\phi$-Selmer group $\Sel_\phi(E)$ to be 
$$
\mathcal{L}_v(E):=\im (E'(\BQ_v)/\phi(E(\BQ_v))) \subseteq H^1(\BQ_v, E[\phi]).
$$ 
To show the claim, it suffices to prove for any $v$, 
$$
\mathcal{L}_v(E^{(M)})=\mathcal{L}_v(E).
$$ 
We now prove the claim by the following four cases. \\
(1) For $v \nmid 2CM\infty$, then both $E$ and $E'$ have good reduction at $v\ne2$ and hence  
$$
\mathcal{L}_v(E^{(M)})=\mathcal{L}_v(E)=H^1_\mathrm{ur}(\BQ_v, E[\phi])
$$ 
is the unramified condition. \\
(2) For $v|M$, the desired equality of local condition at $v$ follows from \cite[Lemma 6.8]{Klagsbrun}. \\
(3) For $v|2C$, by assumption we have $v$ splits in $\BQ(\sqrt{M})$, hence $E^{(M)}$ and $E$ are isomorphic over $\BQ_v$, and $E'^{(M)}$ and $E'$ are isomorphic over $\BQ_v$. The desired equality of local condition at $v$ follows. \\
(4) For $v=\infty$, since $M>0$, we know that $E^{(M)}$ and $E$ are isomorphic over $\mathbb{R}$, and $E'^{(M)}$ and $E'$ are isomorphic over $\mathbb{R}$. The desired equality of local condition at $v$ again follows. \\
This completes the proof of the claim.

Now by \cite[Theorem 6.4]{Klagsbrun}, we have 
$$
\frac{\lvert\Sel_\phi(E)\rvert}{\lvert\Sel_{\hat \phi}(E')\rvert}=\prod_v \frac{\lvert\mathcal{L}_v(E)\rvert}{2},\quad \frac{\lvert\Sel_\phi(E^{(M)})\rvert}{\lvert\Sel_{\hat \phi}(E'^{(M)})\rvert}=\prod_v \frac{\lvert\mathcal{L}_v(E^{(M)})\rvert}{2}.
$$
Since we have shown that $\mathcal{L}_v(E)=\mathcal{L}_v(E^{(M)})$ for every place $v$ of $\BQ$, we obtain 
$$
\frac{\lvert\Sel_\phi(E)\rvert}{\lvert\Sel_{\hat \phi}(E')\rvert}=\frac{\lvert\Sel_\phi(E^{(M)})\rvert}{\lvert\Sel_{\hat \phi}(E'^{(M)})\rvert}.
$$ 
Hence $\Sel_{\hat \phi}(E'^{(M)})\cong \BZ/2\BZ$. Now the well-known exact sequence for $E^{(M)}$ implies  
$$
\dim \Sel_2(E^{(M)})\le \dim \Sel_\phi(E^{(M)})+\dim \Sel_{\hat \phi}(E'^{(M)})=1+1=2.
$$ 
On the other hand, $E^{(M)}(\BQ)[2]\cong \BZ/2\BZ$, so $\dim \Sel_2(E^{(M)})\geq1$. If $\Sha(E^{(M)})[2]$ is finite, then by the Cassels--Tate pairing $\Sha(E^{(M)})[2]$ has square order, hence by the previous bounds it must be 0, as desired. 
\end{proof}

We are now ready to give the proof of Theorem \ref{MainThm-BSD}. 
\begin{thm} (Theorem \ref{MainThm-BSD})
Let $E$ and $M$ be as in Theorem \ref{MainThm}. Assume further that
\begin{enumerate}
  \item $\Sha(E')[2]=0$; 
  \item all primes $\ell$ which divide $2C$ split in $\BQ(\sqrt M)$; 
  \item the $2$-part of the Birch and Swinnerton-Dyer conjecture holds for $E$.
\end{enumerate}
Then the 2-primary component of $\Sha(E^{(M)})$ is zero, and the $2$-part of the Birch and Swinnerton-Dyer conjecture holds for $E^{(M)}$.
\end{thm}
\begin{proof}
If the 2-part of the Birch and Swinnerton-Dyer conjecture holds for $E$, then 
$$
ord_2\left(\frac{\prod_{\ell}c_\ell(E)\cdot \Sha(E)}{\lvert E(\BQ)_\mathrm{tor} \rvert^2}\right)=-1.
$$ 
Since $E(\BQ)[2]\cong \BZ/ 2 \BZ$ and $\Sha(E)[2]$ has square order, we know that $\Sha(E)[2]=0$, $\Sel_2(E)=\BZ/2\BZ$ and $ord_2(\prod_\ell c_\ell(E))=1$. By Theorem \ref{MainThm}, we have 
$$
ord_2(L^{(alg)}E^{(M)},1)=r-1, 
$$
and $\Sha(E^{(M)})$ is finite. The assumptions of Proposition \ref{prop:bsd2} are all satisfied, and hence 
$$
E^{(M)}(\BQ)[2]=\BZ/2\BZ, \quad ord_2(\prod_{p \mid CM}(c_p(E^{(M)}))=r+1, \quad \Sha(E^{(M)})[2]=0.
$$ 
We then have
$$
ord_2\left(\frac{\prod_{p}c_p(E^{M})\cdot \Sha(E^{(M)})}{\lvert E^{(M)}(\BQ)_\mathrm{tor}\rvert^2}\right)=r-1.
$$ 
Therefore, the 2-part of the Birch and Swinnerton-Dyer conjecture holds for $E^{(M)}$.
\end{proof}

\bigskip

\section{Applications}

In this section we will apply Theorem \ref{MainThm} and Theorem \ref{MainThm-BSD} to give some families of quadratic twists of elliptic curves which satisfy the 2-part of the exact Birch--Swinnerton-Dyer formula. In particular, we give a full discussion of quadratic twists of $X_0(14)$, and some analogous examples on the quadratic twists of `$34A1$', `$56B1$', and `$99C1$' (in Cremona's label), for which we will not give the proofs in details since they are similar to the case of $X_0(14)$, and all the numerical examples are verified by `Magma'. Moreover, we also include a family of elliptic curves satisfying the full Birch--Swinnerton-Dyer conjecture. More examples have been included in Wan's paper \cite{Wan}.

In the following, we always denote $A'$ to be the $2$-isogenous curve of a given elliptic curve $A$ defined over $\BQ$. For each square free integer $M$, prime to the conductor of $A$, with $M \equiv 1 \mod 4$, as usual, we define
$$
L^{(alg)}(A^{(M)}, 1) = L(A^{(M)}, 1)/\Omega_{A^{(M)}}.
$$

\smallskip

\subsection{Quadratic twists of \texorpdfstring{$X_0(14)$}{}}

Let $A$ be the modular curve $X_0(14)$, which has genus $1$, and which we view as an elliptic curve by taking $[\infty]$ to be the origin of the group law. It has a minimal Weierstrass equation given by
$$
A: y^2 + xy + y= x^3 + 4x - 6, 
$$
which has non-split multiplicative reduction at $2$. Moreover, $A(\BQ)= \BZ/6\BZ$. The discriminant of $A$ is $-2^6 \cdot 7^3$. Also, a simple computation shows that $\BQ(A[2])=\BQ(\sqrt{-7})$. Writing $L(A, s)$ for the complex $L$-series of $A$, we have
$$
L(A,1)/\Omega_A^+ = 1/6.
$$
Let $q_1,\ldots, q_r$ be $r \geq 0$ distinct primes, which are all $\equiv 1 \mod 4$.

Recall that the $L$-function of an elliptic curve $E$ over $\BQ$ is defined as an infinite Euler product
$$
L(E,s)=\prod_{q \nmid C} (1 - a_q q^{-s} + q^{1-2s})^{-1} \prod_{q \mid C} (1 - a_q q^{-s})^{-1} =: \sum a_n n^{-s},
$$
where
$$
a_q
=
\left\{
\begin{array}{llll}
q+1-\#E(\BF_q) & \hbox{if $E$ has good reduction at $q$,} \\
1              & \hbox{if $E$ has split multiplicative reduction at $q$,} \\
-1             & \hbox{if $E$ has non-split multiplicative reduction at $q$,} \\
0              & \hbox{if $E$ has additive reduction at $q$.}
\end{array}
\right.
$$
Here we give a result of the behavior of the coefficients $a_q$ of the $L$-function of elliptic curve $A$.

\begin{thm}\label{Thm a_q 14}
Let $q$ be an odd prime with $(q,14)=1$. Then we have that
$$
a_2=-1,\ a_7=1,
$$
and
$$
a_q
\equiv
\left\{
\begin{array}{ll}
2 \ \mod 4 & \hbox{if $q \equiv 1 \mod 8$,} \\
2 \ \mod 4 & \hbox{if $q \equiv 3 \mod 8$ and $q$ is inert in $\BQ(\sqrt{-7})$,} \\
2 \ \mod 4 & \hbox{if $q \equiv 5 \mod 8$ and $q$ splits in $\BQ(\sqrt{-7})$,} \\
0 \ \mod 4 & \hbox{if $q \equiv 7 \mod 8$,} \\
0 \ \mod 4 & \hbox{if $q \equiv 3 \mod 8$ and $q$ splits in $\BQ(\sqrt{-7})$,} \\
0 \ \mod 4 & \hbox{if $q \equiv 5 \mod 8$ and $q$ is inert in $\BQ(\sqrt{-7})$.}
\end{array}
\right.
$$
\end{thm}

\begin{proof}
The assertions for $a_2$ and $a_7$ are clear, since $A$ has non-split multiplicative reduction at $2$ and split multiplicative reduction at $7$.

Let $A'$ denote the $2$-isogenous curve of $A$, which has a minimal Weierstrass equation given by
$$
A': y^2+xy+y=x^3-36x-70.
$$
It is easy to get that $\BQ(A'[2])=\BQ(\sqrt{2})$. For $a_q$, first note that the $2$-division field $\BQ(A[2])=\BQ(\sqrt{-7})$ and $\BQ(A'[2])=\BQ(\sqrt{2})$, and we have the same $L$-function of $A$ and $A'$. So we have that $A(\BF_q)[2] \cong \BZ/2\BZ \times \BZ/2\BZ$ when $q$ splits in $\BQ(\sqrt{-7})$, and $A'(\BF_q)[2] \cong \BZ/2\BZ \times \BZ/2\BZ$ when $q$ splits in $\BQ(\sqrt{2})$. Since $A(\BF_q)[2]$ and $A'(\BF_q)[2]$ are subgroups of $A(\BF_q)$ and $A'(\BF_q)$, respectively, we have that $4 \mid \#A(\BF_q)$ and $4 \mid \#A'(\BF_q)$. While $q$ is both inert in $\BQ(\sqrt{2})$ and $\BQ(\sqrt{-7})$, we have that $A(\BF_q)[2] \cong \BZ/2\BZ$. It is easy to compute that $\BQ(\sqrt{2})$ is a subfield of $\BQ(A[4]^*)$, where $A[4]^*$ means any one of the $4$-division points which is deduced from the non-trivial rational $2$-torsion point of $A(\BQ)$. But $q$ is inert in $\BQ(\sqrt{2})$, that means $A(\BF_q)[4] = A(\BF_q)[2] \cong \BZ/2\BZ$. Hence $2 \mid \#A(\BF_q)$, but $4 \nmid \#A(\BF_q)$. Hence
$$
N_q=\#A(\BF_q)
\equiv
\left\{
\begin{array}{ll}
2 \ \mod 4 & \hbox{if $q$ is both inert in $\BQ(\sqrt{2})$ and $\BQ(\sqrt{-7})$,} \\
0 \ \mod 4 & \hbox{if $q$ splits in $\BQ(\sqrt{2})$ or $\BQ(\sqrt{-7})$.}
\end{array}
\right.
$$

Then all the assertions follow by applying $a_q=q+1-N_q$. This completes our proof.
\end{proof}

We then can apply Theorem \ref{MainThm} to get the following result.
\begin{thm}\label{Nonvanishing-14A1}
Let $M$ be any integer of the form $M=q_1 q_2 \cdots q_r$, $r \geq 1$, with $q_1, \ldots, q_r$ arbitrary distinct odd primes all congruent to $5$ modulo $8$, and inert in $\BQ(\sqrt{-7})$. We then have
$$
ord_2(L^{(alg)}(A^{(M)},1))=r-1.
$$
In particular, we have $L(A^{(M)},1) \neq 0$.
\end{thm}
\begin{proof}
According to Theorem \ref{Thm a_q 14}, when $q_i \equiv 3, 5 \mod 8$ and $q_i$ is inert in $\BQ(\sqrt{-7})$, we have $ord_2(N_{q_i})=1$ for $1 \leq i \leq r$. The theorem then follows immediately by Theorem \ref{MainThm-1}.
\end{proof}

We next prove the $2$-part of the Birch and Swinnerton-Dyer conjecture for all the twists $E^{(M)}$ in Theorem \ref{MainThm-BSD}. Note that $A^{(M)}$ has bad additive reduction at all primes dividing $M$. Write $c_q(A^{(M)})$ for the Tamagawa factor of $A^{(M)}$ at a finite odd prime $q \mid M$. We then have that
\begin{equation}\label{A_Tam_q}
ord_2(c_q(A^{(M)})) = ord_2(\# A(\BQ_q)[2]).
\end{equation}
We apply the results in \cite[\S7]{Coates1} on the Tamagawa factors of $A^{(M)}$, and we then get the following result.

\begin{prop}\label{Tam_A-14}
For all odd square-free integers $M$ with $(M,14)=1$, we have (i) $A^{(M)}(\BR)$ has one connected component, (ii) $ord_2(c_2(A^{(M)})) = 1$, $ord_2(c_7(A^{(M)})) = 0$, (iii) $ord_2(c_q(A^{(M)})) = 1$ if $q$ does not split in $\BQ(\sqrt{-7})$, and (iv) $ord_2(c_q(A^{(M)})) = 2$ if $q$ splits in $\BQ(\sqrt{-7})$.
\end{prop}
\begin{proof} Assertion (i) follows immediately from the fact that $\BQ(A[2])=\BQ(\sqrt{-7})$. Assertion (ii) follows easily from Tate's algorithm. The remaining assertions involving odd primes $q$ of bad reduction follow immediately from \eqref{A_Tam_q}, on noting that $A(\BQ_q)[2]$ is of order $2$ or $4$, accordingly as $q$ does not or does split in $\BQ(\sqrt{-7})$, respectively.
\end{proof}

To obtain the $2$-part of the Birch--Swinnerton-Dyer formula, we also have to investigate the $2$-part of $\Sha(A^{(M)})$. If we just apply Theorem \ref{MainThm-BSD}, of course we will get that the $2$-part of the Birch--Swinnerton-Dyer formula holds for a family of quadratic twists, provided both $2$ and $7$ split in $\BQ(\sqrt{M})$, whence $M$ has to have an even number of prime factors. However, a classical 2-descent of quadratic twists of $X_0(14)$ has been carried out earlier by Junhwa Choi, which yields that $\Sha(A^{(M)})[2]$ is trivial, provided that all the prime factors of $M$ are distinct primes congruent to $3,5$ modulo $8$ and inert in $\BQ(\sqrt{-7})$. We then can get the following theorem.

\begin{thm}\label{2BSD-A-14}
Let $M$ be any integer of the form $M=q_1 q_2 \cdots q_r$, $r \geq 1$, with $q_1, \ldots, q_r$ arbitrary distinct odd primes all congruent to $5$ modulo $8$, and inert in $\BQ(\sqrt{-7})$. Then the $2$-part of Birch and Swinnerton-Dyer conjecture is valid for $A^{(M)}$.
\end{thm}
\begin{proof}
Under the assumptions of the theorem, $\Sha(A^{(M)})[2]$ is trivial. Then combining the results of Proposition \ref{Tam_A-14}, we have that $ord_2(\prod_{q \mid M} c_q(A^{(M)})) = r$. Note also that $\#(A(\BQ)[2])= 2$. So we have
$$
ord_2 (\#(\Sha(A^{(M)}))) + ord_2(\prod_p c_p(A^{(M)})) + ord_2 (c_\infty(A^{(M)})) - 2ord_2(\#(A^{(M)}(\BQ))) = r-1.
$$
Hence, the $2$-part of Birch and Swinnerton-Dyer conjecture holds for $A^{(M)}$.
\end{proof}
Here is the beginning of an infinite set of primes $q$ satisfying the conditions in the above theorem:
$$
\mathcal{S}=\{5, 13, 61, 101, 157, 173, 181, 229, 269, 293, 349, 397, \ldots \}.
$$

\smallskip

\subsection{More numerical examples}
For the following three examples, the analogous methods of quadratic twists of $X_0(14)$ would apply, so we will not give the detailed proofs here. 

\subsubsection{Quadratic twists of `$34A1$'}

Let $A$ be the elliptic curve `$34A1$' with the minimal Weierstrass equation given by
$$
A:  y^2 + xy = x^3 - 3x + 1, 
$$
which has split multiplicative reduction at $2$ and $a_2=1$. Moreover, $A(\BQ)= \BZ/6\BZ$ and $L^{(alg)}(A,1)= 1/6$. The discriminant of $A$ is $2^{6} \cdot 17$. Also, a simple computation shows that $\BQ(A[2])=\BQ(\sqrt{17})$ and $\BQ(A'[2])=\BQ(\sqrt{2})$. Here is the beginning of an infinite set of primes $q$ which are congruent to $1$ modulo $4$ and inert in both the fields $\BQ(\sqrt{17})$ and $\BQ(\sqrt{2})$:
$$
\mathcal{S}=\{5, 29, 37, 61, 109, 173, 181, 197, 269, 277, 317, 397, \ldots \}.
$$
Let $M=q_1 q_2 \cdots q_r$, be a product of $r$ distinct primes in $\mathcal{S}$.  We then have 
$$
ord_2(L^{(alg)}(A^{(M)},1))=r-1,
$$ 
and the $2$-part of Birch and Swinnerton-Dyer conjecture is valid for all these twists.

\smallskip

\subsubsection{Quadratic twists of `$56B1$'}

Let $A$ be the elliptic curve `$56B1$' with the minimal Weierstrass equation given by
$$
A: y^2 = x^3 - x^2 - 4, 
$$
which has potentially supersingular reduction at $2$ and $a_2=0$. Moreover, $A(\BQ)= \BZ/2\BZ$ and $L^{(alg)}(A,1)= 1/6$. The discriminant of $A$ is $-2^{10} \cdot 7$. Also, a simple computation shows that $\BQ(A[2])=\BQ(\sqrt{-7})$ and $\BQ(A'[2])=\BQ(\sqrt{2})$. Here is the beginning of an infinite set of primes $q$ which are congruent to $1$ modulo $4$ and inert in both the fields $\BQ(\sqrt{-7})$ and $\BQ(\sqrt{2})$:
$$
\mathcal{S}=\{5, 13, 61, 101, 157, 173, 181, 229, 269, 293, 349, 397, \ldots \}.
$$
Let $M=q_1 q_2 \cdots q_r$, be a product of $r$ distinct primes in $\mathcal{S}$. We then have 
$$
ord_2(L^{(alg)}(A^{(M)},1))=r-1,
$$ 
and the $2$-part of Birch and Swinnerton-Dyer conjecture is valid for all these twists.

\smallskip

\subsubsection{Quadratic twists of `$99C1$'}

Let $A$ be the elliptic curve `$99C1$' with the minimal Weierstrass equation given by
$$
A: y^2 + xy = x^3 - x^2 - 15x + 8, 
$$
which has good reduction at $2$ and $a_2=1$. Moreover, $A(\BQ)= \BZ/2\BZ$ and $L^{(alg)}(A,1)= 1/2$. The discriminant of $A$ is $3^{9} \cdot 11$. Also, a simple computation shows that $\BQ(A[2])=\BQ(\sqrt{33})$ and $\BQ(A'[2])=\BQ(\sqrt{3})$. Here is the beginning of an infinite set of primes $q$ which are congruent to $1$ modulo $4$ and inert in both the fields $\BQ(\sqrt{33})$ and $\BQ(\sqrt{3})$:
$$
\mathcal{S}=\{5, 53, 89, 113, 137, 257, 269, 317, 353, 389, \ldots \}.
$$
Let $M=q_1 q_2 \cdots q_r$, be a product of $r$ distinct primes in $\mathcal{S}$. We then have 
$$
ord_2(L^{(alg)}(A^{(M)},1))=r-1,
$$ 
and the $2$-part of Birch and Swinnerton-Dyer conjecture is valid for all these twists.

\smallskip

\subsection{Examples satisfying the full Birch--Swinnerton-Dyer conjecture} 
Let $A$ be the elliptic curve `$46A1$' with the minimal Weierstrass equation given by
$$
A:  y^2+xy=x^3-x^2-10x-12, 
$$
which has non-split multiplicative reduction at $2$ and $a_2=-1$, $a_3=0$. Moreover, $A(\BQ)= \BZ/2\BZ$ and $L^{(alg)}(A,1)= 1/2$. The discriminant of $A$ is $-2^{10} \cdot 23$. The Tamagawa factors $c_2=2$, $c_{23}=1$. Also, a simple computation shows that $\BQ(A[2])=\BQ(\sqrt{-23})$ and $\BQ(A'[2])=\BQ(\sqrt{2})$. Here is the beginning of an infinite set of primes $q$ which are congruent to $1$ modulo $4$ and inert in both the fields $\BQ(\sqrt{-23})$ and $\BQ(\sqrt{2})$, and satisfy $a_q \neq 0$:
$$
\mathcal{S}=\{5, 37, 53, 61, 149, 157, 181, 229, 293, 373, \ldots \}.
$$
Let $M=q_1 q_2 \cdots q_r$ be a product of $r$ distinct primes in $\mathcal{S}$.  By Theorem \ref{MainThm}, we have $L(A^{(M)},1) \neq 0$, and
$$
ord_2(L^{(alg)}(A^{(M)},1))=r-1.
$$ 
If we carry out a classical $2$-descent on $A^{(M)}$, one shows easily that the $2$-primary component of $\Sha(A^{(M)})$ is zero and $ord_2(c_{q_i})=1$ for $1 \leq i \leq r$, and therefore the $2$-part of the Birch and Swinnerton-Dyer conjecture holds for $E^{(M)}$. Alternatively, we can just apply Theorem \ref{MainThm-BSD}, and take the number of prime factors of $M$, say $r(M)$, to be even, and take $M \equiv 1 \mod 8$. The assumption that both $2$ and $23$ split in $\BQ(\sqrt{M})$ will hold, whence we can also verify the $2$-part of the Birch and Swinnerton-Dyer conjecture. Then combining with the result in \cite[Theorem 9.3]{Wan}, the full Birch and Swinnerton-Dyer conjecture is valid for $A^{(M)}$. Hence the full Birch and Swinnerton-Dyer conjecture is verified for infinitely many elliptic curves.

\bigskip
\bigskip

\noindent Li Cai \\
Yau Mathematical Sciences Center \\
Tsinghua University \\
Beijing 100084 \\
China \\
{\it lcai@mail.tsinghua.edu.cn}

\medskip

\noindent Chao Li \\
Department of Mathematics \\
Columbia University \\
2990 Broadway, New York, NY 10027 \\
U.S.A. \\
{\it chaoli@math.columbia.edu}

\medskip

\noindent Shuai Zhai  \\
Department of Pure Mathematics and Mathematical Statistics \\
University of Cambridge \\
Cambridge CB3 0WB \\
United Kingdom \\
\smallskip
%\& \\
%\smallskip
%Institute for Advanced Research \\
%Qingdao, Shandong 266237 \\
%China \\
{\it S.Zhai@dpmms.cam.ac.uk}

\bigskip

\end{document}